\documentclass{article}
\usepackage[utf8]{inputenc}
\usepackage{fullpage}
\usepackage{amsmath}
\usepackage{amsfonts}
\usepackage{amssymb}
\usepackage{tikz-cd}
\usepackage{color,soul}
\usepackage{array}
\usepackage{zref-savepos}
\usepackage{amsthm}
\usepackage{multirow, bigstrut}
\usepackage[font={small,it}]{caption}

\usepackage{hyperref}
\hypersetup{
    colorlinks=true,
    linkcolor=blue,
    filecolor=magenta,      
    urlcolor=cyan,
}

\usetikzlibrary{calc}
\usepackage{comment}
\usepackage{enumerate}
\usepackage{tikz}
\usetikzlibrary{shapes.geometric}

\newtheorem{theorem}{Theorem}[section]
\newtheorem*{theorem*}{Theorem}
\newtheorem{lemma}[theorem]{Lemma}
\newtheorem{proposition}[theorem]{Proposition}
\newtheorem{corollary}[theorem]{Corollary}

\theoremstyle{definition}
\newtheorem{definition}[theorem]{Definition}
\newtheorem{remark}[theorem]{Remark}

\theoremstyle{plain}

\newcommand{\C}{\mathbb{C}}
\newcommand{\R}{\mathbb{R}}
\newcommand{\N}{\mathbb{N}}
\newcommand{\Z}{\mathbb{Z}}

\DeclareMathOperator{\mesh}{mesh}
\DeclareMathOperator{\lmesh}{lmesh}
\DeclareMathOperator{\cone}{cone}
\DeclareMathOperator*{\E}{\mathbb{E}}

\title{Connecting the $q$-Multiplicative Convolution and the Finite Difference Convolution}
\author{Jonathan Leake \\ Nick Ryder \\}
\date{}

\begin{document}
\maketitle

\begin{abstract}
    In \cite{finitemesh}, Br{\"a}nd{\'e}n, Krasikov, and Shapiro consider root location preservation properties of finite difference operators. To this end, the authors describe a natural polynomial convolution operator and conjecture that it preserves root mesh properties. We prove this conjecture using two methods. The first develops a novel connection between the additive (Walsh) and multiplicative (Grace-Szeg{\"o}) convolutions, which can be generically used to transfer results from multiplicative to additive. We then use this to transfer an analogous result, due to Lamprecht \cite{lamprecht}, which demonstrates logarithmic root mesh preservation properties of a certain $q$-multiplicative convolution operator. The second method proves the result directly using a modification of Lamprecht's proof of the logarithmic root mesh result. We present his original argument in a streamlined fashion and then make the appropriate alterations to apply it to the additive case.
\end{abstract}

\section{Introduction}


Let $\C[x]$ denote the space of polynomials with complex coefficients, and let $\C_n[x]$ denote the subspace of polynomials in $\C[x]$ of degree at most $n$. (We make the analogous definitions for $\R[x]$ and $\R_n[x]$.) The Walsh additive (\cite{walsh}) and Grace-Szeg{\"o} multiplicative (\cite{szego}) polynomial convolutions on $f,g \in \C_n[x]$ have been denoted $\boxplus^n$ and $\boxtimes^n$ respectively (e.g., in \cite{finiteconvolutions}):
\[
    f \boxplus^n g := \sum_{k=0}^n \partial_x^k f \cdot (\partial_x^{n-k} g)(0)
\]
\[
    f \boxtimes^n g := \sum_{k=0}^n \binom{n}{k}^{-1} (-1)^k f_k g_k x^k
\]
This notation is suggestive, as these convolutions can be thought of producing polynomials whose roots are contained in the (Minkowski) sum and product of complex discs containing the roots of the input polynomials. When the inputs have real roots (additive) or non-negative roots (multiplicative), this fact also holds in terms of real intervals containing the roots.

In \cite{finitemesh}, the authors show that the additive convolution can only increase root mesh, which is defined as the minimum absolute difference between any pair of roots of a given polynomial. That is, the mesh of the output polynomial is at least as large as the mesh of either of the input polynomials. They use similar arguments to show that, for polynomials with non-negative roots, the multiplicative convolution can only increase logarithmic root mesh. This is similarly defined as the minimum ratio (greater than 1) between any pair of positive roots of a given polynomial.

Regarding mesh and logarithmic mesh, there are natural generalized convolution operators which also preserve such properties. The first will be called the $q$-multiplicative convolution, and it was shown to preserve logarithmic root mesh of at least $q$ in \cite{lamprecht}. The second will be called the $b$-additive convolution (or finite difference convolution), and the main concern of this paper is to demonstrate that it preserves root mesh of at least $b$.

\subsection*{The $q$-Multiplicative Convolution}

In \cite{lamprecht}, Lamprecht proves logarithmic mesh preservation properties of $q$-multiplicative convolution. This convolution is defined as follows, where $p_k$ an $r_k$ are the coefficients of $p$ and $r$, respectively. (Note that as $q \to 1$ this limits to the classical multiplicative convolution.)
\[
    p \boxtimes_q^n r := \sum_{k=0}^n \binom{n}{k}_q^{-1} q^{-\binom{k}{2}} (-1)^k p_k r_k x^k
\]
Here, $\binom{n}{k}_q$ denotes the $q$-binomial coefficients, to be explicitly defined later. We state his result formally as follows.

\begin{definition}
    Fix $p \in \R[x]$ with all non-negative roots. We say that $p$ is \emph{$q$-log mesh} if the minimum ratio (greater than 1) between any pair of non-zero roots of $p$ is at least $q$. We also say that $p$ is \emph{strictly $q$-log mesh} if the minimum ratio is greater than $q$. In these situations, we write $\lmesh(p) \geq q$ and $\lmesh(p) > q$ respectively.
\end{definition}

\begin{theorem}[Lamprecht]
\label{thm:qlogmesh}
    Let $p$ and $r$ be polynomials of degree at most $n$ such that $\lmesh(p) \geq q$ and $\lmesh(r) \geq q$, for some $q \in (1,\infty)$. Then, $\lmesh(p \boxtimes_q^n r) \geq q$.
\end{theorem}

This result is actually an analogue to an earlier result of Suffridge \cite{suffridge} regarding polynomials with roots on the unit circle. In Suffridge's result, $q$ is taken to be an element of the unit circle, and log mesh translates to mean that the roots are pairwise separated by at least the argument of $q$. Roughly speaking, he obtains the same result for the corresponding $q$-multiplicative convolution. Remarkably, the known proofs of his result (even a proof of Lamprecht) differ fairly substantially from Lamprecht's proof of the above theorem.

Additionally, we note here that Lamprecht uses different notation and conventions in \cite{lamprecht}. In particular, he uses $q \in (0,1)$, considers polynomials $p$ with all non-positive roots, and his definition of $\boxtimes_q^n$ does not include the $(-1)^k$ factor. These differences are generally speaking unsubstantial, but it is worth noting that the arguments of \S\ref{sect:connect} seem to require the $(-1)^k$ factor.

\subsection*{The $b$-Additive Convolution}

In this paper, we show the $b$-additive convolution (or, finite difference convolution) preserves the space of polynomials with root mesh at least $b$. Br{\"a}nd{\'e}n, Krasikov, and Shapiro define the $b$-additive convolution (only for $b=1$) in \cite{finitemesh} as follows. (Note that as $b \to 0$ this limits to the classical additive convolution.)
\[
    p \boxplus_b^n r := \sum_{k=0}^n \Delta_b^k p \cdot (\Delta_b^{n-k} r)(0)
\]
Here, $\Delta_b$ is a finite $b$-difference operator, defined as:
\[
    \Delta_b : p \mapsto \frac{p(x) - p(x-b)}{b}
\]
Our result then solves the first (and second) conjecture stated in \cite{finitemesh}. We state it formally here.

\begin{definition}
    Fix $p \in \R[x]$ with all real roots. We say that $p$ is \emph{$b$-mesh} if the minimum non-negative difference of any pair of roots of $p$ is at least $b$. We also say that $p$ is \emph{strictly $b$-mesh} if this minimum difference is greater than $b$. In these situations, we write $\mesh(p) \geq b$ and $\mesh(p) > b$ respectively.
\end{definition}

\begin{theorem}
\label{thm:main}
    Let $p$ and $r$ be polynomials of degree at most $n$ such that $\mesh(p) \geq b$ and $\mesh(r) \geq b$, for some $b \in (0,\infty)$. Then, $\mesh(p \boxplus_b^n r) \geq b$.
\end{theorem}

As a note, Br{\"a}nd{\'e}n, Krasikov, and Shapiro actually use the \emph{forward} finite difference operator in their definition of the convolution. This is not a problem as our result then differs from their conjecture by a shift of the input polynomials.

\begin{remark}
    Although the $q$-multiplicative and $b$-additive convolutions preserve $q$-log mesh and $b$-mesh respectively, they do not preserve real-rootedness.
\end{remark}

\subsection*{An Analytic Connection}

Our first method of proof of Theorem \ref{thm:main} will demonstrate a way to pass root properties of the $q$-multiplicative convolution to the $b$-additive convolution. As this is interesting in its own right, we state the most general version of this result here.

\begin{theorem}
\label{thm:qbancon}
    Fix $b \geq 0$ and let $p,r$ be polynomials of degree $n$. We have the following, where convergence is uniform on compact sets.
    \[
        \lim_{q \rightarrow 1} (1-q)^n \left[E_{q,b}(p) \boxtimes_{q^b}^n E_{q,b}(r) \right](q^x) = p \boxplus_b^n r
    \]
    Note that for $b=0$, this result pertains to the classical convolutions. 
\end{theorem}

Here, the $E_{q,b}$ are certain linear isomorphisms of $\C[x]$, to be explicitly defined below. Notice that uniform convergence allows us to use Hurwitz' theorem to obtain root properties in the limit of the left-hand side. That is, any information about how the $q$-multiplicative convolution acts on roots will transfer to some statement about how the $b$-additive convolution acts on roots. As it turns out, a special case of Lamprecht's result (Theorem 3 from \cite{lamprecht}) will become our result (Theorem \ref{thm:main}) in the limit. We discuss this transfer process in more detail in \S\ref{sect:connect}.

\subsection*{Extending Lamprecht's Method}

Our second method of proof of Theorem \ref{thm:main} is an extension of the method used by Lamprecht to prove the log mesh result for the $q$-multiplicative convolution. Specifically, he demonstrates that the $q$-multiplicative convolution preserves a root-interlacing property for $q$-log mesh polynomials. More formally he proves the following result which gives Theorem \ref{thm:main} as a corollary.

\begin{theorem}[Lamprecht Interlacing-Preserving]
\label{thm:qinterlacing}
    Let $f,g \in \R_n[x]$ be $q$-log mesh polynomials of degree $n$ with only negative roots. Let $T_g: \R_n[x] \to \R_n[x]$ be the real linear operator defined by $T_g: r \mapsto r \boxtimes_q^n g$. Then, $T_g$ preserves the set of polynomials whose roots interlace the roots of $f$.
\end{theorem}

We achieve an analogous result for the $b$-additive convolution using techniques similar to those found in Lamprecht's paper. We state it formally here.

\begin{theorem}
\label{thm:binterlacing}
    Let $f,g \in \R_n[x]$ be $b$-mesh polynomials of degree $n$. Let $T_g: \R_n[x] \to \R_n[x]$ be the real linear operator defined by $T_g: r \mapsto r \boxplus_b^n g$. Then, $T_g$ preserves the set of polynomials whose roots interlace the roots of $f$.
\end{theorem}

In both cases, mesh and log mesh properties can be shown to be equivalent to root interlacing properties ($f(x)$ interlaces $f(x-b)$ for $b$-mesh, and $f(x)$ interlaces $f(q^{-1}x)$ for $q$-log mesh). The above theorems then immediately imply the desired mesh preservation properties for the respective convolutions. We discuss this further in \S\ref{sect:lamprecht}.











\section{First Proof Method: The Finite Difference Convolution as a Limit of \texorpdfstring{$q$}{q}-Multiplicative Convolutions}
\label{sect:connect}

In what follows we establish a general analytic connection between the multiplicative (Grace-Szego) and additive (Walsh) convolutions on polynomials of degree at most $n$. We then extend this connection to the $q$-multiplicative convolution and the $b$-additive convolution (Theorem \ref{thm:qbancon}). Using this connection, we transfer root information results of the multiplicative convolution ($q$ or classical) to the additive convolution ($b$ or classical). Specifically, we use this connection to prove Theorem \ref{thm:main}, which is the conjecture of Br{\"a}nd{\'e}n, Krasikov, and Shapiro mentioned above.

To begin we state an observation of Vadim Gorin demonstrating an analytic connection in the classical case using matrix formulations of the classical convolutions given in \cite{finiteconvolutions}:
\[
    \chi(A) \boxtimes^n \chi(B) = \E_P \left[\chi(APBP^T)\right]
\]
\[
    \chi(A) \boxplus^n \chi(B) = \E_P \left[\chi(A + PBP^T)\right]
\]
Here, $A$ and $B$ are real symmetric matrices, $\chi$ denotes the characteristic polynomial, and the expectations are taken over all permutation matrices. We then write:
\[
    \begin{split}
        \lim_{t \to 0} t^{-n} \left[\chi(e^{tA}) \boxtimes^n \chi(e^{tB})\right](tx+1) &= \lim_{t \to 0} t^{-n}\E_P\left[\det\left(txI + I - e^{tA}Pe^{tB}P^T\right)\right] \\
            &= \lim_{t \to 0} t^{-n}\E_P\left[\det\left(txI - t(APP^T + PBP^T) + O(t^2)\right)\right] \\
            &= \lim_{t \to 0} \E_P\left[\chi\left(A + PBP^T + O(t)\right)\right] \\
            &= \chi(A) \boxplus^n \chi(B)
    \end{split}
\]
This connection is suggestive and straightforward, but seemingly confinded to the classical case. Therefore, we instead state below a slightly modified (but equivalent) version of this observation for the classical convolutions (Theorem \ref{thm:clsancon}) which we are able to then generalize to the $q$-multiplicative and $b$-additive convolutions.

\subsection{The Classical Convolutions}

We begin by sketching the proof of the connection between the classical additive and multiplicative convolutions. We then state rigorously the more general result for the $q$-multiplicative and $b$-additive convolutions. In this section, many quantities will be defined with $b=0$ in mind (this corresponds to the classical additive convolution), with the more general quantities given in subsequent sections. Further, we will leave the proofs of the lemmas to the generic $b$ case, omitting them here.






To go from the multiplicative world to the additive world, we use a linear map which acts as an exponentiation on roots, and a limiting process which acts as a logarithm. In particular, we will refer to the following algebra endomorphism on $\C[x]$ as our ``exponential map'':
\[
    E_{q,0}: x \mapsto \frac{1-x}{1-q}
\]
In what follows, any limiting process will mean uniform convergence on compact sets in $\C$, unless otherwise specified. This will allow us to extract analytic information about roots using the classical Hurwitz' theorem. In particular, the following result hints at the analytic information provided by the exponential map $E_{q,0}$.

\begin{proposition} \label{prop:clsexponentation}
    We have the following for any $p \in \C[x]$.
    \[
        \lim_{q \rightarrow 1} [E_{q,0}(p)](q^x) = p
    \]
\end{proposition}
\begin{proof}
    We first consider $[E_{q,0}(x)](q^x) = \frac{1-q^x}{1-q}$, for which we obtain the following by the generalized binomial theorem:
    \[
        \lim_{q \to 1} [E_{q,0}(x)](q^x) = \lim_{q \to 1} \frac{1-q^x}{1-q} = \lim_{q \to 1} \sum_{m=1}^\infty \binom{x}{m}(q-1)^{m-1} = \binom{x}{1} = x
    \]
    To show that convergence here is uniform on compact sets, consider the tail for $|x| \leq M$:
    \[
        \begin{split}
            \left|\sum_{m=2}^\infty \binom{x}{m} (q-1)^{m-1}\right| &\leq \sum_{m=0}^\infty \left|(q-1)^{m+1} \cdot \frac{x(x-1) \cdots (x-m-1)}{(m+2)!}\right| \\
                &\leq |q-1| \sum_{m=0}^\infty |q-1|^m \prod_{k=1}^{m+2} \left(1 + \frac{|x|}{k}\right) \\
                &\leq |q-1| (1+M)^2 \sum_{m=0}^\infty \left(|q-1| (1+M)\right)^m \\
                &= \frac{|q-1| (1+M)^2}{1 - |q-1| (1+M)}
        \end{split}
    \]
    This limits to zero as $q \to 1$, which proves the desired convergence.
    
    Since $E_{q,0}$ is an algebra morphism, we can use the fundamental theorem of algebra to complete the proof. Specifically, letting $p(x) = c_0 \prod_k (x - \alpha_k)$ we have:
    \[
        \lim_{q \to 1}\left[E_{q,0}\left(c_0 \prod_k (x-\alpha_k)\right)\right](q^x) = c_0 \prod_k \left(\lim_{q \to 1} [E_{q,0}(x)](q^x) - \alpha_k\right) = c_0 \prod_k (x-\alpha_k)
    \]
\end{proof}

We now state our result in the classical case, which gives an analytic connection between the additive and multiplicative convolutions. As a note, many of the analytic arguments used in the proof of this result will have a flavor similar to that of the proof of Proposition \ref{prop:clsexponentation}.

\begin{theorem} \label{thm:clsancon}
    Let $p,r \in \C[x]$ be of degree at most $n$. We have the following.
    \[
        \lim_{q \rightarrow 1} (1-q)^n \left[E_{q,0}(p) \boxtimes^n E_{q,0}(r) \right](q^x) = p \boxplus^n r
    \]
\end{theorem}

\subsubsection*{Proof Sketch}

We will establish the above identity by calculating it on basis elements. Specifically, we will expand everything into powers of $(1-q)$. To prove the theorem, it then suffices to show that: (1) the negative degree coefficients are all zero, (2) the series has the desired constant term, and (3) the tail of the series converges to zero uniformly on compact sets. Our first step towards establishing this is expanding $q^{kx}$ in terms of powers of $(1-q)$.

\begin{remark}
    Since we will only be considered with behavior for $q$ near $1$, we will use the notation $q \approx 1$ to indicate there exists some $\epsilon > 0$ such that the statement holds for $q \in (1- \epsilon, 1 + \epsilon)$.
\end{remark}

\begin{definition}
   We define a constant, which will help us to simplify the following computations.
    \[
        \alpha_{q,0} := \frac{\ln q}{1-q}
    \]
    Note that $\lim_{q \rightarrow 1} \alpha_{q, 0} = -1$.
\end{definition}

\begin{lemma} \label{lem:clsqexp}
    Fix $k \in \N_0$. For $q \approx 1$, we have the following.
    \[
        q^{kx} = \sum_{m=0}^\infty \frac{x^m}{m!} \alpha_{q,0}^m k^m (1-q)^m
    \]
    For fixed $q \approx 1$, this series has a finite radius of convergence, and this radius approaches infinity as $q \rightarrow 1$.
\end{lemma}

Notice that this is not a true power series in $(1-q)$, as $\alpha_{q,0}$ depends on $q$. Using this, we can calculate the series obtained after plugging in specific basis elements.

\begin{lemma} \label{lem:clsbasiscomp}
    Fix $q \approx 1$ in $\R_+$ and $j,k,n \in \N_0$ such that $0 \leq j \leq k \leq n$. We have the following.
    \[
        (1-q)^n \left[(1-x)^j \boxtimes^n (1-x)^k\right](q^x) = \sum_{m=0}^\infty \frac{x^m}{m!} \alpha_{q,0}^m (1-q)^{n+m-j-k} \sum_{i=0}^j \frac{\binom{j}{i}\binom{k}{i}}{\binom{n}{i}}(-1)^i i^m
    \]
\end{lemma}

We use interpolation arguments to handle the terms of this series, which are combinatorial in nature. In particular we show that this series has no nonzero negative degree terms, as seen in the following.

\begin{proposition}
    Fix $j,k,m,n \in \N_0$ such that $j \leq k$ and $n+m-j-k \leq 0$. We have the following identity.
    \[
        \sum_{i=0}^j \frac{\binom{j}{i}\binom{k}{i}}{\binom{n}{i}}(-1)^i i^m = \left\{
                \begin{array}{ll}
                    (-1)^{n-j-k}\frac{(j)!(k)!}{(n)!} & m = j+k-n \\
                    0 & m < j+k-n
                \end{array}
            \right.
    \]
\end{proposition}

To deal with the tail of the series, we then use crude bounds to get uniform convergence on compact sets.

\begin{lemma}
    Fix $M > 0$, and $j,k,n \in \N_0$ such that $j \leq k \leq n$. For $|x| \leq M$, there exists $\gamma > 0$ such that the following bound holds for $q \in (1-\gamma, 1+\gamma)$.
    \[
        \left|\sum_{m>j+k-n} \frac{x^m}{m!} \alpha_{q,0}^m (1-q)^{n+m-j-k} \sum_{i=0}^j \frac{\binom{j}{i}\binom{k}{i}}{\binom{n}{i}}(-1)^i i^m\right| \leq c_0c_1 \sum_{m=1}^\infty c_2^m |1-q|^m
    \]
    Here, $c_0,c_1,c_2$ are independent of $q$.
\end{lemma}

With this, we can now complete the proof of the theorem by comparing the desired quantity to the constant term in our series in $(1-q)$.

\begin{proof}[Proof of Theorem \ref{thm:clsancon}.]
    For $j,k,n \in \N_0$ such that $0 \leq j \leq k \leq n$, we can combine the above results to obtain the following. Recall that $\lim_{q \rightarrow 1} \alpha_{q,0} = -1$.
    \[
        \lim_{q \rightarrow 1} (1-q)^n \left[(1-x)^j \boxtimes^n (1-x)^k\right](q^x) = \frac{j!k!}{n!(j+k-n)!} x^{j+k-n} = x^j \boxplus^n x^k
    \]
    By symmetry, this demonstrates the desired result on a basis. Therefore, the proof is complete.
\end{proof}

\subsection{General Connection Preliminaries}

We now prove the previous results in more generality, which allows for extension to these generalized convolutions. First though, we give some preliminary notation.

\begin{definition}
    Fix $q \in \R_+$ and $x \in \C$. We define $(x)_q := \frac{1-q^x}{1-q}$. Note that $\lim_{q \rightarrow 1} (x)_q = x$, using the generalized binomial theorem on $q^x = (1+(q-1))^x$.
\end{definition}

Specifically, for any $n \in \Z$, we have:
\[
    (n)_q := \frac{1-q^n}{1-q} = 1 + q + q^2 + \cdots + q^{n-1}
\]
We then extend this notation to $(n)_q! := (n)_q(n-1)_q \cdots (2)_q(1)_q$ and $\binom{n}{k}_q := \frac{(n)_q!}{(k)_q!(n-k)_q!}$. We also define a system of bases of $\C[x]$ which will help us to understand the mesh convolutions.

\begin{definition}
    For $b \geq 0$ and $q \in \R_+$, we define the following bases of $\C[x]$.
    \[
        v_{q,b}^k := \frac{(1-x)(1-q^bx) \cdots (1-q^{(k-1)b}x)}{(1-q)^k}
    \]
    \[
        \nu_b^k := x(x+b)(x+2b) \cdots (x+(k-1)b)
    \]
\end{definition}

We demonstrate the relevance of these bases to the generalized convolutions by giving alternate definitions. Consider a linear map $A_b$ on $\C[x]$ defined via $A_b : \nu_0^k \mapsto \nu_b^k$. That is, $A_b : x^k \mapsto x(x+b)\ldots(x+(k-1)b)$. We can then define the $b$-additive convolution as follows:
\[
    p \boxplus_b^n r := A_b(A_b^{-1}(p) \boxplus^n A_b^{-1}(r))
\]
That is, the $b$-additive convolution is essentially a change of basis of the classical additive convolution. Note that equivalently, one can conjugate $\partial_x$ by $A_b$ to obtain $\Delta_b : p \mapsto \frac{p(x)-p(x-b)}{b}$ which demonstrates the definition of $\boxplus_b^n$ in terms of finite difference operators.

Similarly, the $q$-multiplicative convolution can be seen as a change of basis of the classical multiplicative convolution. Consider a linear map $M_q^{(n)}$ on $\C[x]$ defined via $M_q^{(n)} : \binom{n}{k}x^k \mapsto \binom{n}{k}_q q^{\binom{k}{2}} x^k$, which has the property that $M_{q^b}^{(n)}: (1-x)^n \mapsto (1-q)^n v_{q,b}^n$. We can then define the $q$-multiplicative convolution as follows:
\[
    p \boxtimes_q^n r := M_q^{(n)}\left((M_q^{(n)})^{-1}(p) \boxtimes^n (M_q^{(n)})^{-1}(r)\right)
\]

These bases will be used to simplify the proof of the general analytic connection for the mesh (non-classical) convolutions. In what follows they will play the role that the basis elements $x^k$ and $(1-x)^k$ did in the classical proof sketch above.

\subsection{Main Result for Mesh Convolutions}

We now generalize the results from the classical ($b=0$) setting. First we define a generalized ``exponential map'' using the basis elements defined above. Note for $b > 0$ these are no longer algebra morphisms.
\[
    E_{q,b}: \nu_b^k \mapsto v_{q,b}^k
\]
Notice that specialization to $b = 0$ recovers the original ``exponential map''. Also notice that for any $p$, the roots of $E_{q,b}(p)$ approach 1 as $q \to 1$ (multiply the output polynomial by $(1-q)^{\deg(p)}$). In all that follows, previously stated results can be immediately recovered by setting $b=0$.

\begin{proposition}
    We obtain the same key relation for the generalized exponential maps:
    \[
        \lim_{q \rightarrow 1} [E_{q,b}(p)](q^x) = p
    \]
\end{proposition}
\begin{proof}
    We compute on basis elements, using Proposition \ref{prop:clsexponentation} in the process:
    \[
        \begin{split}
            \lim_{q \rightarrow 1} [E_{q,b}(\nu_b^k)](q^x) &= \lim_{q \rightarrow 1} [v_{q,b}^k](q^x) \\
                &= \lim_{q \rightarrow 1} \frac{(1-q^x)(1-q^{x+b}) \cdots (1-q^{x+(k-1)})}{(1-q)^k} \\
                &= \prod_{j=0}^{k-1} \lim_{q \rightarrow 1} \frac{1-q^{x+jb}}{1-q} = \prod_{j=0}^{k-1} (x+jb) = \nu_b^k
        \end{split}
    \]
\end{proof}

As in Proposition \ref{prop:clsexponentation}, one can interpret the $E_{q,b}$ maps as a way to exponentiate the roots of a polynomial. The inverse to these maps is given in the previous proposition by plugging in an exponential and limiting, which corresponds to taking the logarithm of the roots. This discussion will be made more precise in \S \ref{sect:applications}.

We now state and prove the main result, which gives an analytic link between the $b$-additive and $q$-multiplicative convolutions. We follow the proof sketch of the classical result given above, breaking the following full proof up into more manageable sections.

\newtheorem*{qbancon}{Theorem \ref{thm:qbancon}}
\begin{qbancon}
    Fix $b \geq 0$ and let $p,r$ be polynomials of degree $n$. We have the following.
    \[
        \lim_{q \rightarrow 1} (1-q)^n \left[E_{q,b}(p) \boxtimes_{q^b}^n E_{q,b}(r) \right](q^x) = p \boxplus_b^n r
    \]
\end{qbancon}

\subsubsection*{Series Expansion}

In order to prove this theorem, we first expand the left-hand side of the expression in a series in $(1-q)^m$. As above, this is not quite a power series in $(1-q)^m$ as $\alpha_{q,b}$ (which we now define) depends on $q$.

\begin{definition}
    We define the $b$-version of the $\alpha_{q,0}$ constants as follows.
    \[
        \alpha_{q,b} := \left\{
            \begin{array}{ll}
                \frac{-(b)_{q^{-1}}}{qb} & b > 0 \\
                \frac{\ln q}{1-q} & b = 0
            \end{array}
        \right.
    \]
    Note that $\lim_{b \rightarrow 0} \alpha_{q,b} = \alpha_{q,0}$ for $q \in \R_+$, and $\lim_{q \rightarrow 1} \alpha_{q,b} = -1$ for fixed $b \geq 0$.
\end{definition}

We now need to understand how exponential polynomials in $q$ relate to our basis elements.

\begin{lemma} \label{lem:qbqexp}
    Fix $b \geq 0$, and $k \in \N_0$. For $q \approx 1$ in $\R_+$, we have the following.
    \[
        q^{kx} = \sum_{m=0}^\infty \frac{\nu_b^m}{m!} \alpha_{q,b}^m (k)_{q^{-b}}^m (1-q)^m
    \]
    For fixed $q \approx 1$, this series has a finite radius of convergence, and this radius approaches infinity as $q \rightarrow 1$.
\end{lemma}
\begin{proof}
    For $b > 0$, we use the generalized binomial theorem to compute:
    \[
        \begin{split}
            q^{kx} = (q^{-bk}-1+1)^{-x/b} &= \sum_{m=0}^\infty \frac{(-b)^{-m}}{m!}x(x+b) \cdots (x+b(m-1)) (q^{-bk}-1)^m \\
                &= \sum_{m=0}^\infty \frac{(-b)^{-m} \nu_b^m}{m!} (k)_{q^{-b}}^m (b)_{q^{-1}}^m (q^{-1}-1)^m \\
                &= \sum_{m=0}^\infty \frac{\nu_b^m}{m!} \left(-\frac{(b)_{q^{-1}}}{qb}\right)^m (k)_{q^{-b}}^m (1-q)^m
        \end{split}
    \]
    For $b = 0$, manipulating the Taylor series of $q^{kx} = e^{kx \ln q}$ gives the result.
    
    For fixed $q \approx 1$, let $\delta > 0$ be small enough such that $|\alpha_{q,b}| < \frac{1+\delta}{q}$ and $(k)_{q^{-b}} < k + \delta$. Consider:
    \[
        |\nu_b^m| = |x(x+b)\cdots(x+(m-1)b)| \leq m!(|x|+b)^m
    \]
    From this, we obtain:
    \[
        |q^{kx}| \leq \sum_{m=0}^\infty \left(\frac{(|x|+b)(1+\delta)(k+\delta)}{q}\right)^m |1-q|^m
    \]
    It is then easy to see that the radius of convergence of this series limits to infinity as $q \rightarrow 1$.
\end{proof}

We will now proceed by proving the main result on a basis. To that end, we will prove a number of results related to basis element computations. Most of these are rather tedious and not very illuminating. Perhaps this can be simplified through some more detailed and generalized theory of $q$- and $b$-polynomial operators.

\begin{lemma} \label{lem:qbbasiscomp}
    Fix $q \approx 1$ in $\R_+$, $b \geq 0$, and $j,k,n \in \N_0$ such that $0 \leq j \leq k \leq n$. We have the following.
    \[
        (1-q)^n \left[v_{q,b}^j \boxtimes_{q^b}^n v_{q,b}^k\right](q^x) = \sum_{m=0}^\infty \frac{\nu_b^m}{m!} \alpha_{q,b}^m (1-q)^{n+m-j-k} \sum_{i=0}^j \frac{\binom{j}{i}_{q^b}\binom{k}{i}_{q^b}}{\binom{n}{i}_{q^b}}q^{bi(i-1)/2}(-1)^i (i)_{q^{-b}}^m
    \]
\end{lemma}
\begin{proof}
    We compute:
    \[
        \begin{split}
            (1-q)^n \left[v_{q,b}^j \boxtimes_{q^b}^n v_{q,b}^k\right](q^x) &= (1-q)^{n-j-k} \sum_{i=0}^j \frac{\binom{j}{i}_{q^b}\binom{k}{i}_{q^b}}{\binom{n}{i}_{q^b}}q^{bi(i-1)/2}(-1)^i q^{ix} \\
                &= \sum_{i=0}^j \frac{\binom{j}{i}_{q^b}\binom{k}{i}_{q^b}}{\binom{n}{i}_{q^b}}q^{bi(i-1)/2}(-1)^i \sum_{m=0}^\infty \frac{\nu_b^m}{m!} \alpha_{q,b}^m (1-q)^{n+m-j-k} (i)_{q^{-b}}^m \\
                &= \sum_{m=0}^\infty \frac{\nu_b^m}{m!} \alpha_{q,b}^m (1-q)^{n+m-j-k} \sum_{i=0}^j \frac{\binom{j}{i}_{q^b}\binom{k}{i}_{q^b}}{\binom{n}{i}_{q^b}}q^{bi(i-1)/2}(-1)^i (i)_{q^{-b}}^m
        \end{split}
    \]
\end{proof}

\subsubsection*{Q-Lagrange Interpolation}

To prove convergence in Theorem \ref{thm:qbancon}, we break up the infinite sum of Lemma \ref{lem:qbbasiscomp} into two pieces. For $n+m-j-k \leq 0$, we use an interpolation argument to obtain the following identity. Note that this generalizes a similar identity (for $q=1$) found in \cite{qlagrange_mathof}.

\begin{proposition}
\label{prop:qlagrange}
    Fix $q \approx 1$, $b \geq 0$, and $j,k,m,n \in \N_0$ such that $j \leq k$ and $n+m-k \leq j$. We have the following identity.
    \[
        \sum_{i=0}^j \frac{\binom{j}{i}_{q^b}\binom{k}{i}_{q^b}}{\binom{n}{i}_{q^b}}q^{bi(i-1)/2}(-1)^i (i)_{q^{-b}}^m = \left\{
                \begin{array}{ll}
                    (-1)^{n-j-k}q^{b\binom{n}{2}-b\binom{j}{2}-b\binom{k}{2}}\frac{(j)_{q^b}!(k)_{q^b}!}{(n)_{q^b}!} & m = j+k-n \\
                    0 & m < j+k-n
                \end{array}
            \right.
    \]
\end{proposition}

We first give a lemma, and then the proof of the proposition will follow. Let $[t^j]p(t)$ denote the coefficient of $p$ corresponding to the monomial $t^j$.

\begin{lemma}\label{interpolate}
    Fix $p \in \C_j[x]$. We have the following identity:
    \[
        (-1)^j q^{-\binom{j}{2}} \cdot [t^j]p(t) = \sum_{i=0}^j p((i)_{q^{-1}}) \frac{(-1)^{i}}{(i)_q! (j-i)_q!} q^{\binom{i}{2}}
    \]
\end{lemma}
\begin{proof}
    Using Lagrange interpolation, the following holds for any polynomial of degree at most $j$:
    \[
        [t^j]p(t) = \sum_{i=0}^j p((i)_q) \frac{(-1)^{j-i}}{(i)_q! (j-i)_q!} q^{-\binom{i}{2}} q^{-i(j-i)}
    \]
    Using the identity $(i)_{q^{-1}}! = q^{-\binom{i}{2}} (i)_q!$ (via $(i)_{q^{-1}} = q^{-i+1} (i)_q$) and replacing $q$ by $q^{-1}$ gives:
    \[
        \begin{split}
            [t^j]p(t) &= \sum_{i=0}^j p((i)_{q^{-1}}) \frac{(-1)^{j-i}}{(i)_q! (j-i)_q!} q^{2\binom{i}{2}} q^{i(j-i)} q^{\binom{j-i}{2}} \\
                &= \sum_{i=0}^j p((i)_{q^{-1}}) \frac{(-1)^{j-i}}{(i)_q! (j-i)_q!} q^{\binom{i}{2}} q^{\binom{j}{2}}
        \end{split}
    \]
    The result follows.
\end{proof}

\begin{proof}[Proof of Proposition \ref{prop:qlagrange}]
    Consider the polynomial $p(t) = t^m((n)_{q^{-b}}-t)((n-1)_{q^{-b}}-t) \cdots ((k+1)_{q^{-b}}-t)$, which is of degree $m+n-k \leq j$. So, $[t^j]p(t) = (-1)^{n-k}\delta_{m=j+k-n}$. Also, recall the identity $(i)_{q^{-b}}! = q^{-b\binom{i}{2}} (i)_{q^b}!$. Using the previous lemma and replacing $q$ by $q^b$, we obtain:
    \[
        \begin{split}
            (-1)^{j} q^{-b\binom{j}{2}} \cdot (-1)^{n-k}\delta_{m=j+k-n} &= \sum_{i=0}^{j} p((i)_{q^{-b}}) \frac{(-1)^{i}}{(i)_{q^b}! (j-i)_{q^b}!} q^{b\binom{i}{2}} \\
                &= \sum_{i=0}^{j} q^{-b i (n-k)}  \frac{(n-i)_{q^{-b}}!}{(k-i)_{q^{-b}}!}\frac{(-1)^{i}}{(i)_{q^b}! (j-i)_{q^b}!} q^{b\binom{i}{2}}(i)_{q^{-b}}^m \\
                &= \sum_{i=0}^{j} q^{b \binom{k}{2}} q^{-b \binom{n}{2}} \frac{(n-i)_{q^{b}}!}{(k-i)_{q^{b}}!}\frac{(-1)^{i}}{(i)_{q^b}! (j-i)_{q^b}!} q^{b\binom{i}{2}}(i)_{q^{-b}}^m \\
                &= \sum_{i=0}^j q^{b \binom{k}{2}} q^{-b \binom{n}{2}} \frac{(n)_{q^b}}{(j)_{q^b} (k)_{q^b}} \cdot \frac{\binom{j}{i}_{q^b} \binom{k}{i}_{q^b}}{\binom{n}{i}_{q^b}} q^{b\binom{i}{2}} (-1)^i (i)_{q^{-b}}^m
        \end{split}
    \]
    The result follows.
\end{proof}

\subsubsection*{Tail of the Series}

For $n+m-j-k > 0$, we show that the tail of the infinite series in Lemma \ref{lem:qbbasiscomp} is bounded by a geometric series in $\epsilon \rightarrow 0$ as $q \rightarrow 1$. The proof, is somewhat similar to the discussion of convergence in the proof of Lemma \ref{lem:qbqexp}.

\begin{lemma}
    Fix $b \geq 0$, $M > 0$, and $j,k,n \in \N_0$ such that $j \leq k \leq n$. For $|x| \leq M$, there exists $\gamma > 0$ such that the following bound holds for $q \in (1-\gamma, 1+\gamma)$.
    \[
        \left|\sum_{m>j+k-n} \frac{\nu_b^m}{m!} \alpha_{q,b}^m (1-q)^{n+m-j-k} \sum_{i=0}^j \frac{\binom{j}{i}_{q^b}\binom{k}{i}_{q^b}}{\binom{n}{i}_{q^b}}q^{bi(i-1)/2}(-1)^i (i)_{q^{-b}}^m\right| \leq c_0c_1 \sum_{m=1}^\infty c_2^m |1-q|^m
    \]
    Here, $c_0,c_1,c_2$ are independent of $q$.
\end{lemma}
\begin{proof}
Fix $n+m-j-k > 0$ with $j \leq k \leq n$ and $q \approx 1$. We have the following bound, where $c_0$ is some positive constant independent of $q$:
\[
    \left|\sum_{i=0}^j \frac{\binom{j}{i}_{q^b}\binom{k}{i}_{q^b}}{\binom{n}{i}_{q^b}}q^{bi(i-1)/2}(-1)^i (i)_{q^{-b}}^m\right| \leq \sum_{i=0}^j c_0 (i+\delta)^m \leq c_0 (n+\delta)^{m+1}
\]
For $|x| \leq M$, we have:
\[
    |\nu_b^m| = |x(x+b)\cdots(x+(m-1)b)| \leq |M(M+b)\cdots(M+(m-1)b)| \leq m!(M+b)^m
\]
This then implies the following bound on the tail. Let $c_1 := (n+\delta) \big[(1+\delta)(M+b)(n+\delta)\big]^{j+k-n}$ and $c_2 := (1+\delta)(M+b)(n+\delta)$, where small $\delta > 0$ is needed to deal with limiting details.
\[
    \begin{split}
        \Bigg|\sum_{m=j+k+1-n}^\infty \frac{\nu_b^m}{m!} \alpha_{q,b}^m &(1-q)^{n+m-j-k} \sum_{i=0}^j \frac{\binom{j}{i}_{q^b}\binom{k}{i}_{q^b}}{\binom{n}{i}_{q^b}}q^{bi(i-1)/2}(-1)^i (i)_{q^{-b}}^m\Bigg|  \\
            &\leq \sum_{m=j+k+1-n}^\infty \left|\frac{\nu_b^m}{m!} \alpha_{q,b}^m (1-q)^{n+m-j-k}\right| c_0(n+\delta)^{m+1} \\
            &\leq c_0\sum_{m=j+k+1-n}^\infty (1+\delta)^m(M+b)^m |1-q|^{n+m-j-k}(n+\delta)^{m+1} \\
            &\leq c_0c_1 \sum_{m=j+k+1-n}^\infty \big[(1+\delta)(M+b)(n+\delta)|1-q|\big]^{n+m-j-k} \\
            &= c_0c_1 \sum_{m=1}^\infty c_2^m |1-q|^m
    \end{split}
\]
So, for any $\epsilon > 0$ we can select $q$ close enough to 1 such that $|1-q| < \frac{\epsilon}{c_2}$. This implies the above series is geometric with terms bounded by $\epsilon^m$.

\end{proof}

The above lemma in particular demonstrates that the tail of the series in Lemma \ref{lem:qbbasiscomp} converges to 0 uniformly on compact sets. With this, we can now complete the proof of the theorem.

\begin{proof}[Proof of Theorem \ref{thm:qbancon}.]
    For $j,k,n \in \N_0$ such that $0 \leq j \leq k \leq n$, we can combine the above results. When we expand our limit as a sum of powers of $(1-q)$, we have shown that everything limits to zero except for the constant term. Recall that $\lim_{q \rightarrow 1} \alpha_{q,b} = -1$.
    \[
    \begin{split}
        \lim_{q \rightarrow 1} (1-q)^n \left[v_{q,b}^j \boxtimes_{q^b}^n v_{q,b}^k\right](q^x) &= \lim_{q \to 1} \frac{\nu_{b}^{j+k-n} \alpha_{q,b}^{j+k-n}}{(j+k-n)!}(-1)^{j+k-n}q^{b\binom{n}{2}-b\binom{j}{2}-b\binom{k}{2}}\frac{(j)_{q^b}!(k)_{q^b}!}{(n)_{q^b}!} \\
        &= \frac{j!k!}{n!(j+k-n)!} \nu_b^{j+k-n} \\
        &= \nu_b^j \boxplus_b^n \nu_b^k\\
    \end{split}
    \]
    By symmetry, this demonstrates the desired result on a basis. Therefore, the proof is complete.
\end{proof}

\subsection{Applications To Previous Results}
\label{sect:applications}

The main motivation for the multiplicative to additive convolution connection was to be able to relate seemingly analogous root information results. The following table outlines the results we proceed to connect.
\[
    \begin{tabular}{ c | c }
        Additive Convolution & Multiplicative Convolution \\ \hline
        Preserves Real Rooted Polynomials & Preserves Positive Rooted Polynomials \\
        Additive Max Root Triangle Inequality & Multiplicative Max Root Triangle Inequality \\
        Preserves $b-$Mesh & Preserves $q-$Logarithmic Mesh \\ 
    \end{tabular}
\]
All of these connections have a similar flavor, and rely on the following elementary facts about exponential polynomials. We say $f(x) = \sum_{k=0}^n c_k q^{kx}$ is an \emph{exponential polynomial of degree $n$ with base $q$}. A real number $x$ is a root of $f$ if and only if $q^x$ is a root of $\sum_{k=0}^n c_k x^k$. Because of this we can bootstrap the fundamental theorem of algebra. 

\begin{definition}
    We call $\left\{x \in \C : \frac{-\pi}{|\ln(q)|} < \Im(x) < \frac{\pi}{|\ln(q)|}\right\}$ the \emph{principal strip} (with respect to $q$). Let $p(q^x)$ be an exponential polynomial of degree $n$ with base $q$. The number of roots of $p(q^x)$ in the principal strip is the same as the number of roots of $p$ in $\C \setminus (-\infty,0]$. We call the roots in the principal strip the \emph{principal roots}.
\end{definition}

\begin{lemma}
    The principal roots of $E_{q,b}(p)[q^x]$ converge to the roots of $p$ as $q \to 1$. In particular, $E_{q,b}(p)[q^x]$ has $\deg(p)$ principal roots for $q \approx 1$.
\end{lemma}
\begin{proof}
    This follows from the fact that, as $q \to 1$, $E_{q,b}(p)[q^x]$ converges uniformly on compact sets to $p$ and the principal strip grows towards the whole plane. 
\end{proof}

We can analyze the behavior of this convergence when $p$ is real rooted with distinct roots.

\begin{lemma} \label{lem:realrootconverge}
    Suppose $p$ is real with real distinct roots. For $q \approx 1$, we have that $E_{q,b}(p)[q^x]$ has principal roots which are real and distinct (and converging to the roots of $p$).
\end{lemma}
\begin{proof}
    Since $p$ has real coefficients, the roots of $E_{q,b}(p)[q^x]$ are either real or come in conjugate pairs. (Consider the fact that $q^{\overline{x}} = \overline{q^x}$.) If $p$ has real distinct roots, the previous lemma implies the principal roots of $E_{q,b}(p)[q^x]$ have distinct real part for $q$ close enough to 1. Therefore, the principal roots of $E_{q,b}(p)[q^x]$ must all be real.
\end{proof}

If we exponentiate (with base $q$) the principal roots of $E_{q,b}(p)[q^x]$, we get the roots of $E_{q,b}(p)$. So if the principal roots of $E_{q,b}(p)[q^x]$ are real, then the roots of $E_{q,b}(p)$ are positive. Considering the above results, this means that $E_{q,b}$ maps polynomials with distinct real roots to polynomials with distinct positive roots for $q \approx 1$. (In fact, the roots will be near 1.)

\subsubsection*{Root Preservation}

The most classical results about the roots are the following:

\begin{theorem*}[Root Preservation]
\leavevmode
\begin{itemize}
    \item If $p,r \in \mathbb{R}_n[x]$ have positive roots, then $p \boxtimes^n r$ has positive roots.
    \item If $p,r \in \mathbb{R}_n[x]$ have real roots, then $p \boxplus^n r$ has real roots.
\end{itemize}
\end{theorem*}

\noindent Neither of these results are particularly hard to prove, but showing how the additive result follows from the multiplicative serves as a prime example of how our theorem connects results on the roots.

\begin{proof}[Proof of Additive from Multiplicative]
    We can reduce to showing that the additive convolution preserves real rooted polynomials with distinct roots since the closure of polynomials with distinct real roots is all real rooted polynomials.
    
    By Lemma \ref{lem:realrootconverge}, the roots of $E_{q,0}(p)$ are real, distinct, and exponentials of the principal roots of $E_{q,0}[p](q^x)$ for $q \approx 1$. This implies that $E_{q,0}(p)$ has positive real roots. By the multiplicative result, $E_{q,0}(p) \boxtimes^n E_{q,0}(r)$ has positive real roots, and therefore $[E_{q,0}(p) \boxtimes^n E_{q,0}(r)](q^x)$ has real principal roots. By our main result, $(1-q)^n[E_{q,0}(p) \boxtimes^n E_{q,0}(r)](q^x)$ converges to $p \boxplus^n r$. The real-rootedness of $[E_{q,0}(p) \boxtimes^n E_{q,0}(r)](q^x)$ for $q \approx 1$ then implies $p \boxplus^n r$ is real-rooted.
\end{proof}

\subsubsection*{Triangle Inequality}

The next classical theorem relates to the max root of a given polynomial. Given a real-rooted polynomial $p$, let $\lambda({p})$ denote its max root. Given an exponential polynomial $f$ with principal roots all real, let $\lambda({f})$ denote the largest principal root of $f$. Also, denote $\exp_q(\alpha) := q^\alpha$.

\begin{theorem*}[Triangle Inequalities]
    \leavevmode
    \begin{itemize}
        \item Given positive-rooted polynomials $p, r$ we have $\lambda({p \boxtimes^n r}) \leq \lambda({p}) \cdot \lambda({r})$
        \item Given real-rooted polynomials $p, r$ we have $\lambda({p \boxplus^n r}) \leq \lambda({p}) + \lambda({r})$
    \end{itemize}
\end{theorem*}

\noindent As before, neither of these have particularly complicated proofs, but we can use the multiplicative result to deduce the additive result in the following.

\begin{proof}[Proof of Additive from Multiplicative]
    As in the previous proof, we can reduce to showing that the result holds for $p,r$ with distinct roots. For this proof, we only consider $q > 1$.
    
    By Lemma \ref{lem:realrootconverge}, we have that the roots of $E_{q,0}(p)$ are real, distinct, and exponentials of the principal roots of $E_{q,0}[p](q^x)$ for $q \approx 1$. This implies the roots of $E_{q,0}(p)$ are positive for $q \approx 1$. Additionally, notice that $\exp_q(\lambda(f(q^x))) = \lambda(f(p))$ whenever $f$ is positive-rooted. From the multiplicative result and the fact that $\boxtimes^n$ preserves positive-rootednes, we have the following for $q \approx 1$:
    \[
        \begin{split}
            \exp_q(\lambda([E_{q,0}(p) \boxtimes^n E_{q,0}(r)](q^x))) &= \lambda({E_{q,0}(p) \boxtimes^n E_{q,0}(r)}) \\
                &\leq \lambda({E_{q,0}(p)}) \cdot \lambda({E_{q,0}(r)}) \\
                &= \exp_q(\lambda({E_{q,0}[p](q^x)}) + \lambda({E_{q,0}[r](q^x)}))
        \end{split}
    \]
    Therefore, $\lambda([E_{q,0}(p) \boxtimes^n E_{q,0}(r)](q^x)) \leq \lambda(E_{q,0}[p](q^x)) + \lambda(E_{q,0}[r](q^x))$. By our main result, $(1-q)^n[E_{q,0}(p) \boxtimes^n E_{q,0}(r)](q^x)$ converges to $p \boxplus^n r$, and therefore $\lambda([E_{q,0}(p) \boxtimes^n E_{q,0}(r)](q^x))$ converges to $\lambda(p \boxplus^n r)$. Similarly $\lambda({E_{q,0}[p](q^x)})$ converges to $\lambda({p})$, and the result follows.
\end{proof}

\subsubsection*{Application to Mesh Preservation Conjecture}

Recall the log mesh result of Lamprecht in \cite{lamprecht} regarding the $q$-multiplicative convolution.


\newtheorem*{qlogmesh}{Theorem \ref{thm:qlogmesh}}
\begin{qlogmesh}
    Fix $q > 1$. Given positive-rooted polynomials $p,r \in \mathbb{R}_n[x]$ with $\lmesh(p), \lmesh(r) \geq q$, we have:
    \[
        \lmesh(p \boxtimes_q^n r) \geq q
    \]
\end{qlogmesh}

In \cite{finitemesh}, Br{\"a}nd{\'e}n, Krasikov, and Shapiro conjectured the analogous result for the $b$-additive convolution (for $b=1$). Using our connection we will confirm this conjecture:

\newtheorem*{main}{Theorem \ref{thm:main}}
\begin{main}
    Given real-rooted polynomials $p,r \in \mathbb{R}_n[x]$ with $\mesh(p), \mesh(r) \geq b$, we have:
    \[
        mesh(p \boxplus_b^n r) \geq b
    \]
\end{main}
\begin{proof}
    We will prove this claim for polynomials $p, r$ with $\mesh(p), \mesh(r) > b$. Since we can approximate any polynomial with $\mesh(p) = b$ by polynomials with larger mesh, the result then follows.
    
    By Lemma \ref{lem:realrootconverge}, $E_{q,b}[p](q^x)$ has real roots which converge to the roots of $p$ for $q \approx 1$. Since the roots of $p$ satisfy $\mesh(p) > b$, the principal roots of $E_{q,b}[p](q^x)$ will have mesh greater than $b$ for $q \approx 1$. Further, $\lmesh(E_{q,b}(p)) = \exp_q(\mesh(E_{q,b}[p](q^x))) > q^b$. (All of this discussion holds for $r$ as well.) By our main result, we have:
    \[
        \lim_{q \rightarrow 1} (1-q)^n \left[E_{q,b}(p) \boxtimes_{q^b}^n E_{q,b}(r) \right](q^x) = p \boxplus_b^n r
    \]
    By the previous theorem, the $q^b$-multiplicative convolution of $E_{q,b}(p)$ and $E_{q,b}(r)$ has logarithmic mesh at least $q^b$. Precomposition by $q^x$ then yields an exponential polynomial with mesh (of the principal roots) at least $b$. The principal roots of this exponential polynomial then converge to $p \boxplus_b^n r$, and hence $p \boxplus_b^n r$ has mesh at least $b$.
\end{proof}

\section{Second Proof Method: A Direct Proof using Interlacing}
\label{sect:lamprecht}

While the previous framework generically transferred Lamprecht's multiplicative result to prove the conjectured result in the additive realm, one might desire a direct proof to gain insight on the underlying structure of the convolution. In what follows, we first outline the preliminary knowledge required to understand a special case of Lamprecht's argument. Then we outline his approach in the multiplicative case and extend this approach to the additive realm to prove the desired conjecture.

\subsection{Interlacing Preserving Operators}

Given $f,g \in \R[x]$, we say $f \ll g$ iff $f'g-fg' \leq 0$ iff $\left(\frac{f}{g}\right)' \leq 0$ wherever defined. Further, we say $f \ll g$ \emph{strictly} iff $f'g-fg' < 0$. Additionally, it is well known that $f$ and $g$ are real-rooted with interlacing roots iff $f \ll g$ or $g \ll f$. Further, $f$ and $g$ have strictly interlacing roots (no shared roots) iff $f \ll g$ strictly or $g \ll f$ strictly.

Let $\lambda_f$ denote the largest root of $f$. If $f$ and $g$ are monic, then $f \ll g$ implies $\lambda_f \leq \lambda_g$. We give a short proof of this now. If $f$ has a double root at $\lambda_f$, then interlacing implies $g(\lambda_f) = 0$, and therefore $\lambda_f \leq \lambda_g$. Otherwise, consider that $f \ll g$ implies $f'(\lambda_f) \cdot g(\lambda_f) = (f'g-fg')(\lambda_f) \leq 0$. Since $f$ is monic, we have $f'(\lambda_f) > 0$ which in turn implies $g(\lambda_f) \leq 0$. Since $g$ is monic, this implies the result. Note that this further implies that if $f$ and $g$ are monic and $\deg(f) = \deg(g)-1$, then $f$ and $g$ have interlacing roots iff $f \ll g$.

Another classical result allows us to combine interlacing relations. If $f \ll g$ and $f \ll h$, then $f \ll ag+bh$ for any $a,b \in \R_+$. A similar result holds if $g \ll f$ and $h \ll f$. Note also that $f \ll g$ iff $g \ll -f$, and that $af \ll bf$ for all $a,b \in \R$. Finally, the Hermite-Biehler theorem says $af+bg$ is real-rooted for all $a,b \in \R$ iff either $f \ll g$ or $g \ll f$.

\begin{remark}
    A polynomial $f$ with non-negative roots is $q$-log mesh if and only if $f \ll f(q^{-1}x)$ and strictly $q$-log mesh if and only if $f \ll f(q^{-1}x)$ strictly (for $q > 1$). Similarly, a polynomial $f$ with real roots is $b$-mesh if and only if $f \ll f(x-b)$ and strictly $b$-mesh if and only if $f \ll f(x-b)$ strictly (for $b > 0$).
\end{remark}

Now let $f$ and $g$ be of degree $n$, and suppose $f$ has $n$ simple real roots, $\alpha_1,...,\alpha_n$. By partial fraction decomposition, we have:
\[
    \frac{g(x)}{f(x)} = c + \sum_{k=1}^n \frac{c_{\alpha_k}}{x-\alpha_k}
\]
Denoting $f_{\alpha_k}(x) := \frac{f(x)}{(x-\alpha_k)}$, this implies:
\[
    g(x) = cf(x) + \sum_{k=1}^n c_{\alpha_k} f_{\alpha_k}(x)
\]
If $g(\alpha_k) = 0$, then $c_{\alpha_k}=0$. Otherwise we compute:
\[
    c_{\alpha_k} = \lim_{x \to \alpha_k} \frac{(x-\alpha_k)g(x)}{f(x)} = \left[\frac{f'(\alpha_k)}{g(\alpha_k)}\right]^{-1} = \left[\left(\frac{f}{g}\right)'(\alpha_k)\right]^{-1}
\]
This leads to the first result, which is a classical one.

\begin{proposition}
    Fix $f,g \in \R_n[x]$. Suppose $f$ is monic and has $n$ simple real roots, $\alpha_1,...,\alpha_n$. Consider the decomposition:
    \[
        g(x) = cf(x) + \sum_{k=1}^n c_{\alpha_k} f_{\alpha_k}(x)
    \]
    Then, $g \ll f$ iff $c_{\alpha_k} \geq 0$ for all $k$, and $f \ll g$ iff $c_{\alpha_k} \leq 0$ for all $k$.
\end{proposition}
\begin{proof}
    $(\Rightarrow)$. If $f \ll g$, then $\left(\frac{f}{g}\right)' \leq 0$. This implies $c_{\alpha_k} \leq 0$ for all $k$. If $g \ll f$, then $f \ll -g$. The same argument implies $c_{\alpha_k} \geq 0$ for all $k$.
    
    $(\Leftarrow)$. By the above argument, $cf \ll f$ and $f_{\alpha_k} \ll f$ for all $k$. So, if $c_{\alpha_k} \geq 0$ for all $k$, then:
    \[
        g = cf + \sum_{k=1}^n c_{\alpha_k} f_{\alpha_k}(x) \ll f
    \]
    A similar argument works to show $f \ll g$ if $c_{\alpha_k} \leq 0$ for all $k$.
\end{proof}

There is actually another way to state this result, in terms of cones of polynomials. Let $\cone(f_1,...,f_m)$ denote the closure of the positive cone generated by the polynomials $f_1,...,f_m$.

\begin{corollary}
    Let $f \in \R_n[x]$ be a monic polynomial with $n$ simple roots, $\alpha_1,...,\alpha_n$. Then, $\{g \in \R_n[x] : g \ll f\} = \cone(f,-f,f_{\alpha_1},...,f_{\alpha_n})$ and $\{g \in \R_n[x] : f \ll g\} = \cone(f,-f,-f_{\alpha_1},...,-f_{\alpha_n})$.
\end{corollary}
\begin{proof}
    Since any $g \in \R_n[x]$ can be written as a linear combination of $f,f_{\alpha_1},...,f_{\alpha_n}$, the result follows from the previous proposition.
\end{proof}

This immediately yields the following result concerning linear operators preserving certain interlacing relations.

\begin{definition}
    Given a real linear operator $T: \R_n[x] \to \R[x]$, and a real-rooted polynomial $f$, we say that $T$ \emph{preserves interlacing with respect to $f$} if $g \ll f$ implies $T[g] \ll T[f]$ and $f \ll g$ implies $T[f] \ll T[g]$ for all $g \in \R_n[x]$
\end{definition}

\begin{corollary}\label{interlacing_op_cor}
    Fix a real linear operator $T: \R_n[x] \to \R[x]$, and fix $f \in \R_n[x]$. Suppose $f$ is monic with $n$ simple roots, $\alpha_1,...,\alpha_n$, and that $T[f_{\alpha_k}] \ll T[f]$ for all $k$. Then, $T$ preserves interlacing with respect to $f$.
\end{corollary}

\subsection{Lamprecht's Approach}

In what follows, we follow Lamprecht's approach to proving that the space of $q$-log mesh polynomials is preserved by the $q$-mutliplicative convolution. Here, we are only interested in proving this result for $q$-log mesh polynomials with non-negative roots, which simplifies the proof. (Lamprecht demonstrates this result for a more general class of polynomials.) The main structure of the proof is: (1) establish properties of two distinguished polar derivatives, (2) show how these derivatives relate to the $q$-multiplicative convolution, and (3) use this to prove that the $q$-multiplicative convolution preserves certain interlacing properties. In the next section, we will emulate this method for $b$-mesh polynomials and the $b$-additive convolution.



\subsubsection*{\texorpdfstring{$q$}{q}-Polar Derivatives}

In \cite{lamprecht}, Lamprecht defines the following $q$-derivative operators, which generalize the operators $\partial_x$ and $-\partial_y$ on homogeneous polynomials. Here, $q > 1$ is always assumed, as above. (As a note, Lamprecht uses the $\Delta$ symbol for these derivatives, and actually gives different definitions as his convention is $q \in (0,1)$.)
\[
    (d_{q,n} f)(x) := \frac{f(qx) - f(x)}{q^{1-n}(q^n-1)x}
    ~~~~~~~~~~~~~~~
    (d_{q,n}^* f)(x) := \frac{f(qx) - q^nf(x)}{q^n-1}
\]
He then goes on to show that these ``derivative'' operators have similar preservation properties to that of the usual derivatives. In particular, he obtains the following.

\begin{proposition}\label{q_deriv_prop}
    The operators $d_{q,n}: \R_n[x] \to \R_{n-1}[x]$ and $d_{q,n}^*: \R_n[x] \to \R_{n-1}[x]$ preserve the space of $q$-log mesh polynomials and the space of strictly $q$-log mesh polynomials. Further, we have that $d_{q,n}f \ll f$ and $d_{q,n}^*f \ll d_{q,n}f$.
\end{proposition}

The above result is actually spread across a number of results in Lamprecht's paper. We omit the proof for now, referring the reader to \cite{lamprecht}.

\subsubsection*{Recursive Identities}
Lamprecht then determines the following identities, which are crucial to his inductive proof of the main result of this section. Fix $f \in \R_{n-1}[x]$ and $g \in \R_n[x]$.
\[
    f \boxtimes_q^n g = f \boxtimes_q^{n-1} d_{q,n}^*g
    ~~~~~~~~~~~~~~~
    (xf) \boxtimes_q^n g = x(f \boxtimes_q^{n-1} d_{q,n}g)
\]

\subsubsection*{Lamprecht's Proof}
With this, we now state an interesting result about interlacing preservation of the $q$-convolution operator. We will then derive the main result as a corollary.

\newtheorem*{qinterlacing}{Theorem \ref{thm:qinterlacing}}
\begin{qinterlacing}[Lamprecht Interlacing-Preserving]
    Let $f,g \in \R_n[x]$ be $q$-log mesh polynomials of degree $n$ with only positive roots. Let $T_g: \R_n[x] \to \R_n[x]$ be the real linear operator defined by $T_g: r \mapsto r \boxtimes_q^n g$. Then, $T_g$ preserves interlacing with respect to $f$.
\end{qinterlacing}
\begin{proof}
    We prove the theorem by induction. For $n=1$ the result is straightforward, as $\boxtimes_q^1 \equiv \boxtimes^1$. For $m>1$, we inductively assume that the result holds for $n=m-1$. By Corollary \ref{interlacing_op_cor} and the fact that $f$ has $n$ simple roots, we only need to show that $T_g[f_{\alpha_k}] \ll T_g[f]$ for all roots $\alpha_k$ of $f$. That is, we want to show $f_{\alpha_k} \boxtimes_q^m g \ll f \boxtimes_q^m g$ for all $k$.
    
    By Proposition \ref{q_deriv_prop}, we have that $d_{q,m}g$ and $d_{q,m}^*g$ are $q$-log mesh and $d_{q,m}^*g \ll d_{q,m}g$. Further, $d_{q,m}g$ and $d_{q,m}^*g$ are of degree $m-1$ and have no roots at 0. The inductive hypothesis and symmetry of $\boxtimes_q^n$ then imply:
    \[
        f_{\alpha_k} \boxtimes_q^{m-1} d^*_{q,m}g \ll f_{\alpha_k} \boxtimes_q^{m-1} d_{q,m}g
    \]
    The fact that these polynomial have leading coefficients with the same sign means that the max root of $f_{\alpha_k} \boxtimes_q^{m-1} d_{q,m}g$ is larger than that of $f_{\alpha_k} \boxtimes_q^{m-1} d^*_{q,m}g$. Further, since all roots are positive we obtain:
    \[
        f_{\alpha_k} \boxtimes_q^{m-1} d_{q,m}^*g \ll x(f_{\alpha_k} \boxtimes_q^{m-1} d_{q,m}g)
    \]
    By properties of $\ll$, this gives:
    \[
        f_{\alpha_k} \boxtimes_q^{m-1} d_{q,m}^*g \ll x(f_{\alpha_k} \boxtimes_q^{m-1} d_{q,m}g) - \alpha_k(f_{\alpha_k} \boxtimes_q^{m-1} d_{q,m}^*g)
    \]
    By the above identities and the fact that $f(x) = (x-\alpha_k)f_{\alpha_k}(x)$, this is equivalent to $f_{\alpha_k} \boxtimes_q^m g \ll f \boxtimes_q^m g$.
\end{proof}

\newtheorem*{qlogmeshcor}{Corollary \ref{thm:qlogmesh}}
\begin{qlogmeshcor}
    Let $f,g \in \R_n[x]$ be $q$-log mesh polynomials (with non-negative roots), not necessarily of degree $n$. Then, $f \boxtimes_q^n g$ is $q$-log mesh.
\end{qlogmeshcor}
\begin{proof}
    First suppose $f,g$ are of degree $n$ with only positive roots. Since $f \ll f(q^{-1}x)$, the previous theorem implies:
    \[
        f \boxtimes_q^n g \ll f(q^{-1}x) \boxtimes_q^n g = (f \boxtimes_q^n g)(q^{-1}x)
    \]
    That is, $f \boxtimes_q^n g$ is $q$-log mesh.
    
    Otherwise, suppose $f$ is of degree $m_f \leq n$ with $z_f$ roots at 0 and $g$ is of degree $m_g \leq n$ with $z_g$ roots at zero. Intuitively, we now add roots ``near 0 and $\infty$'' and limit. Let new polynomials $F$ and $G$ be given as follows:
    \[
        F(x) := f(x) \cdot x^{-z_f}\prod_{j=1}^{z_f} \left(x - \frac{1}{\alpha_j}\right) \cdot \prod_{j=m_f+1}^n \left(\frac{x}{\alpha_j} - 1\right)
    \]
    \[
        G(x) := g(x) \cdot x^{-z_g}\prod_{j=1}^{z_g} \left(x - \frac{1}{\beta_j}\right) \cdot \prod_{j=m_g+1}^n \left(\frac{x}{\beta_j} - 1\right)
    \]
    Here, $\alpha_j$ and $\beta_j$ are any large positive numbers such that $F$ and $G$ are $q$-log mesh polynomials of degree $n$. By the previous argument, $F \boxtimes_q^n G$ is $q$-log mesh. Letting $\alpha_j$ and $\beta_j$ limit to $\infty$ (while preserving $q$-log mesh) implies $F \boxtimes_q^n G \to f \boxtimes_q^n g$ root-wise, which implies $f \boxtimes_q^n g$ is $q$-log mesh.
\end{proof}

Lamprecht is actually able to remove the degree $n$ with positive roots restriction earlier in the line of argument, albeit at the cost of a more complicated proof. We have elected here to take the simpler route. He also proves similar results for a class of $q$-log mesh polynomials with possibly negative roots, which we omit here.

\subsection{\texorpdfstring{$b$}{b}-Additive Convolution}

The main structure of Lamprecht's argument revolves around the two ``polar'' $q$-derivatives, $d_{q,n}$ and $d_{q,n}^*$. The key properties of these derivatives are: (1) they preserve the space of $q$-log mesh polynomials, and (2) they recursively work well with the definition of the $q$-multiplicative convolution. So, when extending this argument to the $b$-additive convolution we face an immediate problem: there is only one natural derivative which preserves the space of $b$-mesh polynomials. This stems from the fact that 0 and $\infty$ have special roles in the $q$-multiplicative world, whereas only $\infty$ is special in the $b$-additive world. The key idea we introduce then is that given a fixed $b$-mesh polynomial $f$, we can pick a polar derivative with pole ``close enough to $\infty$'' so that it maps $f$ to a $b$-mesh polynomial. The fact that we use a different polar derivative for each fixed input $f$ does not affect the proof method.

We now give a few facts about the finite difference operator $\Delta_b$, which plays a crucial role in the definition of the $b$-additive convolution. Recall its definition:
\[
    (\Delta_{b,n} f)(x) \equiv (\Delta_b f)(x) := \frac{f(x) - f(x-b)}{b}
\]
(We use the notation $\Delta_{b,n}$ when we want to restrict the domain to $\R_n[x]$, as in Proposition \ref{b_deriv_prop} below.) This operator acts on rising factorial polynomials as the usual derivative acts on monomials. That is, for all $k$:
\[
    \Delta_b x(x+b)\cdots(x+(k-1)b) = kx(x+b)\cdots(x+(k-2)b)
\]
This operator has preservation properties similar to that of the usual derivative and the $q$-derivatives. The following result, along with many others regarding mesh and log-mesh polynomials, can be found in \cite{fisk}.

\begin{proposition}\label{b_deriv_prop}
    The operator $\Delta_{b,n}: \R_n[x] \to \R_{n-1}[x]$ preserves the space of $b$-mesh polynomials, and the space of strictly $b$-mesh polynomials. Further, we have $\Delta_b f \ll f$ and $\Delta_b f \ll f(x-b)$. If $f$ is strictly $b$-mesh, then these interlacings are strict.
\end{proposition}

\subsubsection*{Finding Another Polar Derivative}

We now define another ``derivative-like'' operator that is meant to generalize $\partial_y$ and $d_{q,n}^*$. Notice that unlike $\Delta_b$, this operation depends on $n$.
\[
    (\Delta^*_{b,n} f)(x) := nf(x-b) - (x-b)\Delta_b f(x)
\]
Unfortunately, this operator does \emph{not} preserve $b$-mesh. However, it does generalize other important properties of $\partial_y$. In particular, it maps $\R_n[x]$ to $\R_{n-1}[x]$, and as $b \to 0$ it limits to $\partial_y f$, the polar derivative of $f$ with respect to 0. Further, we have the following results.

\begin{lemma}\label{b_dy_basis_lemma}
    Fix $f \in \R_n[x]$ and write $f = \sum_{k=0}^n a_k x(x+b)\cdots(x+(k-1)b)$. Then:
    \[
        (\Delta_{b,n}^* f)(x+b) = \sum_{k=0}^{n-1} (n-k)a_k x(x+b)\cdots(x+(k-1)b)
    \]
\end{lemma}

This next lemma is a generalization of the corollary following it.

\begin{lemma}
    Fix monic polynomials $f,g \in \R[x]$ of degree $m$ and $m-1$, respectively, such that $g$ is strictly $b$-mesh and $g \ll f$ strictly. Denote $h_{a,t}(x) := af(x) - (x-t)g(x)$ for $a \geq 1$ and $t > 0$. For all $t$ large enough, we have $g \ll h_{a,t}$ strictly, $h_{a,t} \ll f$ strictly, and $h_{a,t}$ is strictly $b$-mesh.
\end{lemma}
\begin{proof}
    Denote $h_{a,t}(x) := af(x) - (x-t)g(x)$. Since $f,g$ are monic, we have that $h_{a,t}$ is of degree at most $m$ with positive leading coefficient (for large $t$ if $a=1$). Further, if $\alpha_1 < \cdots < \alpha_{m-1}$ are the roots of $g$ and $\beta_1 < \cdots < \beta_m$ are the roots of $f$, then $g \ll f$ strictly and $t$ large implies:
    \[
        \begin{array}{cc}
            h_{a,t}(\alpha_{m-1}) = af(\alpha_{m-1}) < 0
            ~~~~~&~~~~~
            h_{a,t}(\beta_m) = -(\beta_m-t)g(\beta_m) > 0 \\
            h_{a,t}(\alpha_{m-2}) = af(\alpha_{m-2}) > 0
            ~~~~~&~~~~~
            h_{a,t}(\beta_{m-1}) = -(\beta_{m-1}-t)g(\beta_{m-1}) < 0 \\
            h_{a,t}(\alpha_{m-3}) = af(\alpha_{m-3}) < 0
            ~~~~~&~~~~~
            h_{a,t}(\beta_{m-2}) = -(\beta_{m-2}-t)g(\beta_{m-2}) > 0 \\
            \vdots ~~~~~&~~~~~ \vdots \\
        \end{array}
    \]
    The alternating signs imply $h_{a,t}$ has an odd number of roots in the interval $(\alpha_k, \beta_{k+1})$ and an even number of roots in the interval $(\beta_k,\alpha_k)$ for all $1 \leq k \leq m-1$. Since the degree of $h_{a,t}$ is at most $m$, each of these intervals must contain exactly one root and zero roots, respectively. If $h_{a,t}$ is of degree $m$, then it has one more root which must be real since $h_{a,t} \in \R[x]$. Additionally, since $h_{a,t}$ has positive leading coefficient, this last root must lie in the interval $(-\infty,\beta_1)$ (and not in $(\beta_m,\infty)$). Therefore, $g \ll h_{a,t}$ strictly and $h_{a,t} \ll f$ strictly.
    
    Finally, $h_{a,t} \to g$ as $t \to \infty$ coefficient-wise, and so therefore also in terms of the zeros. This means that the root in the interval $(\alpha_k,\beta_{k+1})$ will limit to $\alpha_k$ from above (for all $k$). Further, the possible root in the interval $(-\infty,\beta_1)$ will then limit to $-\infty$, as $g$ is of degree $m-1$. Since $g$ is strictly $b$-mesh, this implies $h_{a,t}$ is also strictly $b$-mesh for large enough $t$.
\end{proof}

\begin{corollary}\label{b_pol_deriv_cor}
    Let $f \in \R_n[x]$ be strictly $b$-mesh. Then for all $t > 0$ large enough, we have that $(t\Delta_{b,n} + \Delta^*_{b,n})f$ is strictly $b$-mesh and $\Delta_{b,n} f \ll (t\Delta_{b,n} + \Delta^*_{b,n})f$ strictly.
\end{corollary}
\begin{proof}
    Consider $(t\Delta_{b,n} + \Delta^*_{b,n})f = nf(x-b) - (x-b-t)\Delta_{b,n}f$. Note that $\Delta_{b,n}f \in \R_{n-1}[x]$ is strictly $b$-mesh and of degree one less than $f$, and $\Delta_{b,n}f \ll f(x-b)$ strictly by Proposition \ref{b_deriv_prop}. Now assume WLOG that $f$ is monic and of degree at least 1. Letting $c$ denote the leading coefficient of $\Delta_{b,n}f$, we have $1 \leq c \leq n$. We can then write:
    \[
        \frac{1}{c}(t\Delta_{b,n} + \Delta^*_{b,n})f = \frac{n}{c}f(x-b) - (x-b-t)\frac{\Delta_{b,n}f}{c}
    \]
    Applying the previous lemma to $f(x-b)$ and $\frac{\Delta_{b,n}f}{c}$ with $a = \frac{n}{c}$ gives the result.
\end{proof}

This corollary says that $t\Delta_{b,n} + \Delta^*_{b,n}$ preserves $b$-mesh, even though $\Delta_{b,n}^*$ does not. The operator $t\Delta_{b,n} + \Delta^*_{b,n}$ can be thought of as the polar derivative with respect to $t$, since by limiting $b \to 0$ we obtain the classical polar derivative.

\subsubsection*{Recursive Identities}

The $\Delta_{b,n}^*$ operator is also required to obtain $b$-additive convolution identities similar to Lamprecht's given above.

\begin{lemma}
    Fix $f \in \R_{n-1}[x]$ and $g \in \R_n[x]$. We have:
    \[
        f \boxplus_b^n g = f \boxplus_b^{n-1} \Delta_{b,n}g
        ~~~~~~~~~~~~~~~
        (xf) \boxplus_b^n g = x(f \boxplus_b^{n-1} \Delta_{b,n}g) + f \boxplus_b^{n-1} \Delta^*_{b,n} g
    \]
\end{lemma}
\begin{proof}
    The first identity is straightforward from the definition of $\boxplus_b^n$. As for the second, we compute:
    \[
        \Delta_b^k(xf) = \Delta_b^{k-1}(x\Delta_b f + f(x-b)) = \cdots = x\Delta_b^k f + k\Delta_b^{k-1}f(x-b)
    \]
    Notice that $\Delta_b$ commutes with shifting, so this is unambiguous. This implies:
    \[
        \begin{split}
            (xf) \boxplus_b^n g &= \sum_{k=0}^n (x\Delta_b^k f + k\Delta_b^{k-1}f(x-b)) \cdot (\Delta_b^{n-k} g)(0) \\
                &= x(f \boxplus_b^{n-1} \Delta_{b,n}g) + \sum_{k=1}^n k\Delta_b^{k-1}f(x-b) \cdot (\Delta_b^{n-k} g)(0) \\
                &= x(f \boxplus_b^{n-1} \Delta_{b,n}g) + \sum_{k=0}^{n-1} \Delta_b^{n-1-k}f(x-b) \cdot ((n-k)\Delta_b^{k} g)(0) \\
                &= x(f \boxplus_b^{n-1} \Delta_{b,n}g) + f(x-b) \boxplus_b^{n-1} (\Delta_{b,n}^* g)(x+b)
        \end{split}
    \]
    The last step of the above computation uses Lemma \ref{b_dy_basis_lemma} and the fact that $(\Delta_b^k g)(0)$ picks out the coefficient corresponding to the $k^\text{th}$ rising factorial term. Finally:
    \[
        f(x-b) \boxplus_b^{n-1} (\Delta_{b,n}^* g)(x+b) = (f \boxplus_b^{n-1} (\Delta_{b,n}^* g)(x+b))(x-b) = f \boxplus_b^{n-1} \Delta_{b,n}^* g
    \]
    This implies the second identity.
\end{proof}

With this we can now emulate Lamprecht's proof to prove interlacing preserving properties of the $b$-additive convolution.

\subsubsection*{Lamprecht-Style Proof}

\newtheorem*{binterlacing}{Theorem \ref{thm:binterlacing}}
\begin{binterlacing}
    Let $f,g \in \R_n[x]$ be strictly $b$-mesh polynomials of degree $n$. Let $T_g: \R_n[x] \to \R_n[x]$ be the real linear operator defined by $T_g: r \mapsto r \boxplus_b^n g$. Then, $T_g$ preserves interlacing with respect to $f$.
\end{binterlacing}
\begin{proof}
    We prove the theorem by induction. For $n=1$ the result is straightforward, as $\boxplus_b^1 \equiv \boxplus^1$. For $m>1$, we inductively assume that the result holds for $n=m-1$. By Corollary \ref{interlacing_op_cor}, we only need to show that $T_g[f_{\alpha_k}] \ll T_g[f]$ for all roots $\alpha_k$ of $f$. That is, we want to show $f_{\alpha_k} \boxplus_b^m g \ll f \boxplus_b^m g$ for all $k$.
    
    By Proposition \ref{b_deriv_prop} and Corollary \ref{b_pol_deriv_cor}, we have that $\Delta_{b,m}g$ and $(t\Delta_{b,m}+\Delta_{b,m}^*)g$ are strictly $b$-mesh and $\Delta_{b,m}g \ll (t\Delta_{b,m}+\Delta_{b,m}^*)g$ strictly for large enough $t$. Further, $\Delta_{b,m}g$ and $(t\Delta_{b,m}+\Delta_{b,m}^*)g$ are of degree $m-1$. The inductive hypothesis and symmetry of $\boxplus_b^n$ then imply:
    \[
        f_{\alpha_k} \boxplus_b^{m-1} \Delta_{b,m}g \ll f_{\alpha_k} \boxplus_b^{m-1} (t\Delta_{b,m}+\Delta_{b,m}^*)g
    \]
    Considering the discussion near the beginning of the paper, we also have:
    \[
        f_{\alpha_k} \boxplus_b^{m-1} \Delta_{b,m}g \ll (x-\alpha_k-t)(f_{\alpha_k} \boxplus_b^{m-1} \Delta_{b,m}g)
    \]
    By properties of $\ll$, this gives:
    \[
        \begin{split}
            f_{\alpha_k} \boxplus_b^{m-1} \Delta_{b,m}g &\ll (x-\alpha_k-t)(f_{\alpha_k} \boxplus_b^{m-1} \Delta_{b,m}g) + f_{\alpha_k} \boxplus_b^{m-1} (t\Delta_{b,m}+\Delta_{b,m}^*)g \\
                &= (x-\alpha_k)(f_{\alpha_k} \boxplus_b^{m-1} \Delta_{b,m}g) + f_{\alpha_k} \boxplus_b^{m-1} \Delta_{b,m}^* g
        \end{split}
    \]
    By the above identities and the fact that $f(x) = (x-\alpha_k)f_{\alpha_k}(x)$, this is equivalent to $f_{\alpha_k} \boxplus_b^m g \ll f \boxplus_b^m g$.
\end{proof}

\newtheorem*{maincor}{Corollary \ref{thm:main}}
\begin{maincor}
    Let $f,g \in \R_n[x]$ be strictly $b$-mesh polynomials. Then, $f \boxplus_b^n g$ is $b$-mesh.
\end{maincor}
\begin{proof}
    First suppose $f,g$ are of degree $n$. Since $f \ll f(x-b)$ strictly, the previous theorem implies the following:
    \[
        f \boxplus_b^n g \ll f(x-b) \boxplus_b^n g = (f \boxplus_b^n g)(x-b)
    \]
    That is, $f \boxplus_b^n g$ is $b$-mesh.
    
    If $f,g$ are not of degree $n$, then the result follows by adding new roots and limiting them to $\infty$, in a fashion similar to the proof of Corollary \ref{thm:qlogmesh} given above.
\end{proof}

\section{Conclusion}

\subsubsection*{Extensions of other classical convolution results}
In this paper we investigate connections between the additive and multiplicative convolutions and their mesh generalizations. Looking forward, it is natural to look at other results in the classical case and ask for mesh generalizations. To the authors knowledge, there are two classical results which have been extended to mesh analogues: in \cite{finitemesh}, the authors explore extensions of the Hermite-Poulain theorem to the $1$-mesh world, and in \cite{lamprecht}, Lamprecht extends classical results for the multiplicative convolution to the $q$-log mesh world. 

An important related result in the classical case is the triangle inequality, which we discuss in \S\ref{sect:applications}. To our knowledge, there is not a known generalization of the triangle inequality to the mesh and log mesh cases. If one could establish such a result for the $q$-multiplicative convolution, it would automatically extend to the $b$-additive convolution using our analytic connection. 

\subsubsection*{Extensions of other $q-$multiplicative convolution results}

In addition to log mesh preservation, Lamprecht proves other results about the $q$-multiplicative convolution. Here we comment on these and their relation to the mesh world.

Beyond the finite degree case, Lamprecht discusses the extension of Laguerre-Polya functions to the $q$-multiplicative world, and then establishes a $q$-version of Polya-Schur multiplier sequences via a power series convolution. Since we are not aware of analogous power series results for the classical additive convolution, we have not explored the connections to the $b$-additive case.

Additionally, Lamprecht classifies log-concave sequences in terms of $q$-log mesh polynomials using the Hadamard product and a limiting argument. There might be an analogue result in the mesh world for concave sequences, but it is unclear what would take the place of the Hadamard product.



Lamprecht details the classes of polynomials that the $q$-multiplicative convolution preserves. Most of these results come from the presence of two poles in the $q$-multiplicative case, yielding derivative operators which preserve negative- and positive-rootedness respectively. The $b$-additive case does not have such complications. In our simplification of Lamprecht's argument, we assume the input polynomials to be generic (strictly $b$-mesh), and then limit to obtain the result for all $b$-mesh polynomials. By keeping track of boundary case information, Lamprecht is able to get more precise results about boundary elements of the space of $b$-mesh polynomials. We believe it is likely possible to emulate this in the above proof with more bookkeeping.

\subsubsection*{The analytic connection applied to other known classical results}

There are other results known about the classical multiplicative convolution which we believe could be transferred to the additive convolution using our generic framework. Specifically in \cite{finiteconvolutions}, Marcus, Spielman, and Srivastava establish a refinement of the triangle inequality for both the additive and multiplicative convolutions. These refinements parallel the well studied transforms from free probability theory. We have not yet worked out the details of this connection. 

\subsubsection*{Further directions for the generic analytic connection}

Finally, it is worth nothing that our analytic connection can only transfer results about the multiplicative convolution to the additive convolution. The main obstruction is finding the appropriate analogue to the exponential map. The following limiting connection between exponential polynomials and polynomials motivated our investigation:
\[
    \lim_{q\to1} \frac{1 - q^x}{1-q} = x
\]
Finding the appropriate ``logarithmic analogue'' could yield a way to pass results from the additive convolution to the multiplicative convolution. That said, some heuristic evidence suggests that such an analogue might not exist.

Above all, our analytic connection still remains rather mysterious. We suspect that there exists a more general theory which provides better intuition for this multiplicative-to-additive connection. While developing this connection, we found multiple candidate exponential maps which experimentally worked. We settled on the ones introduced in this paper due to their relatively nice combinatorial properties. Ideally, an alternative approach would avoid proving the result on a basis and better explain the role of these ``exponential maps''.

\bibliographystyle{amsalpha}
\bibliography{bibliography}

\providecommand{\bysame}{\leavevmode\hbox to3em{\hrulefill}\thinspace}
\providecommand{\MR}{\relax\ifhmode\unskip\space\fi MR }
\providecommand{\MRhref}[2]{%
  \href{http://www.ams.org/mathscinet-getitem?mr=#1}{#2}
}
\providecommand{\href}[2]{#2}
\begin{thebibliography}{BKS16}

\bibitem[BKS16]{finitemesh}
Petter Br{\"a}nd{\'e}n, Ilia Krasikov, and Boris Shapiro, \emph{Elements of
  p{\'o}lya-schur theory in the finite difference setting}, Proceedings of the
  American Mathematical Society \textbf{144, 11} (2016), 4831--4843.

\bibitem[Fis06]{fisk}
S~Fisk, \emph{Polynomials, roots, and interlacing},
  \href{https://arxiv.org/abs/math/0612833}{https://arxiv.org/abs/math/0612833}
  (2006), xx+700.

\bibitem[Lam16]{lamprecht}
Martin Lamprecht, \emph{Suffridge's convolution theorem for polynomials and
  entire functions having only real zeros}, Advances in Mathematics
  \textbf{288} (2016), 426--463.

\bibitem[MSS15]{finiteconvolutions}
A~Marcus, D~Spielman, and N~Srivastava, \emph{Finite free convolutions of
  polynomials}, arXiv preprint (2015), arXiv:1504.00350.

\bibitem[Pet]{qlagrange_mathof}
Fedor (https://mathoverflow.net/users/4312/fedor-petrov) Petrov,
  \emph{Combinatorial identity with connection coefficients and falling
  factorial $\langle i x\rangle_n$}, MathOverflow,
  URL:https://mathoverflow.net/q/256199 (version: 2016-12-02).

\bibitem[Suf73]{suffridge}
T~Suffridge, \emph{Starlike functions as limits of polynomials}, Lecture Notes
  in Mathematics \textbf{505} (1973), 164--203.

\bibitem[Sze22]{szego}
G~Szeg{\"o}, \emph{Bemerkungen zu einem satz von j h grace {\"u}ber die wurzeln
  algebraischer gleichungen}, Mathematische Zeitschrift (1922), 28--55.

\bibitem[Wal22]{walsh}
JL~Walsh, \emph{On the location of the roots of certain types of polynomials},
  Transactions of the American Mathematical Society (1922), 163--180.

\end{thebibliography}



\end{document}